\documentclass{amsart}
\usepackage{amssymb,latexsym}
\theoremstyle{plain}
\newtheorem{theorem}{Theorem}

\newtheorem{lemma}{Lemma}
\theoremstyle{definition}

\newtheorem{remark}{Remark}

\DeclareMathOperator{\red}{red}

\DeclareMathOperator{\reg}{reg}

\newcommand{\enm}[1]{\ensuremath{#1}}          %

\newcommand{\cal}[1]{\mathcal{#1}}

\newcommand{\ZZ}{\enm{\mathbb{Z}}}

\newcommand{\PP}{\enm{\mathbb{P}}}

\newcommand{\Ii}{\enm{\cal{I}}}

\newcommand{\Oo}{\enm{\cal{O}}}

\begin{document}

\title[Severi variety]
{On the irreducibility of the Severi variety of nodal curves in a smooth surface}
\author{Edoardo Ballico}
\address{Dept. of Mathematics\\
 University of Trento\\
38123 Povo (TN), Italy}
\email{ballico@science.unitn.it}
\thanks{The author was partially supported by MIUR and GNSAGA of INdAM (Italy).}

\begin{abstract}
Let $X$ be a smooth projective surface and $L\in \mathrm{Pic}(X)$. We prove that if $L$ is $(2k-1)$-spanned, then the set
$\tilde{V}_k(L)$ of all nodal and irreducible $D\in |L|$ with exactly $k$ nodes is irreducible. The set $\tilde{V}_k(L)$
is an open subset of a Severi variety of $|L|$, the full Severi variety parametrizing all integral $D\in |L|$ with geometric
genus
$g(L)-k$.
\end{abstract}

\maketitle

\section{Introduction}
Let $X$ be a smooth projective variety and $L\in \mathrm{Pic}(X)$. We recall that a zero-dimensional scheme $Z\subset X$ is
said to be \emph{curvilinear} if for each $p\in Z_{\red}$ the Zariski tangent space of $Z$ at $p$ has at most dimension $1$
(this is equivalent to the existence of a smooth curve $C\subset X$ such that $Z\subset C$). We recall that the pair
$(X,L)$ is said to be
$k$-spanned (resp. $k$-very ample) if for each curvilinear zero-dimensional scheme (resp. each zero-dimensional scheme)
$Z\subset X$ the restriction map $H^0(X,L)\to H^0(Z,L_{|Z})$ is surjective (\cite[pages 225 and 278]{bs}). Obviously $k$-very ampleness implies
$k$-spannedness, but in a very particular case (but a very important case: smooth K3 surfaces and smooth Enriques surfaces) the two
notions coincide (\cite[Theorems 1.1 and 1.2]{kn}). Now assume that $S$ is a smooth surface and that $L$ is very ample. The
genus
$g(L)$ of
$L$ is the arithmetic genus of any $D\in |L|$, i.e. $g(L) = 1 + (L^2+\omega _X\cdot L)/2$ (adjunction formula). Let $V_k(L)$
denote
the set of all integral $D\in |L|$ with geometric genus $g(L)-k$. Let $\tilde{V}_k(L)$ be the set of all integral $D\in
|L|$ with exactly $k$ ordinary nodes as its only singularities. Obviously $\tilde{V}_k(X)$ is an open subset of $V_k(L)$. In
the terminology of \cite{cd} we have $V^L_{g(L)-k} = V_k(L)$ and $V^{L,k} = \tilde{V}_k(L)$. In many cases to check the
irreducibility of
$V_k(L)$ (which of course only holds under certain assumptions) one first proves that $\tilde{V}_k(X)$ is irreducible and then
one proves that
$\tilde{V}_k(X)$ is dense in
$V_k(L)$. 

In this note we prove the following result.

\begin{theorem}\label{i1}
Let $X$ be a smooth projective surface. Fix a very ample $L\in \mathrm{Pic}(X)$ and a positive integer $k$. If $L$ is
$(2k-1)$-spanned, then $\tilde{V}_k(L)$ is either irreducible of dimension $h^0(X,L) - 1-k$ or
$\tilde{V}_k(L) =\emptyset$.
\end{theorem}

Theorem \ref{i1} improves \cite[Theorem 9.1]{k1}, which says that $\tilde{V}_k(L)$ is irreducible if $L$ is $(3k-1)$-very
ample (so from $(3k-1)$ to $(2k-1)$, a huge improvement). Then M. Kemeny apply \cite[Theorem 9.1]{k1} to the case in which
$X$ is a K3 surface with $\mathrm{Pic}(X)\cong \ZZ$, but he also remark that other methods give stronger results (in
particular \cite[Corollary 2.6]{ke1}). For another case for a K3 surface with
$\mathrm{Pic}(X)\cong \ZZ$ see \cite{cd}. In
\cite{ke1, ke2} T. Keitel gives several cases in which
$\tilde{V}_k(L)$ is irreducible, but
in all cases $\mathrm{Pic}(X)$ is asked to be controlled (e.g. $X$ the product of two Pic-independent curves or $\mathrm{Pic}(X)\cong \ZZ$ or
$X$ has no curve with negative self-intersection). Theorem \ref{i1} does not cover \cite[Theorem 1]{cd}, because (with the
notation of \cite{cd}) it would require $g\ge 4\delta -1$, while \cite[Theorem 1]{cd} only requires $g\ge 4\delta -3$. We quote
\cite{ds} for some cases in which $\tilde{V}_d(L)$ is dense in $V_d(L)$, which is thus irreducible. For tools used in
Keitel's proofs and that can be probably used in several other cases, see \cite{gls1, gls2, gls3, gls4, ke1, ke2}.

We thank the referee for help in the organization of the paper.

\section{The proof}
We recall the following lemma, often called \emph{the curvilinear lemma}, which is due to K. Chandler (\cite[Cor. 2.4]{c0},
\cite[Lemma 4]{c1}) and often used to compute the dimensions of secant varieties of varieties embedded in a projective space.

\begin{lemma}\label{a1}
Fix an integral projective variety, a very ample line bundle $L\in
\mathrm{Pic}(X)$ and a finite set $S\subset X_{\reg}$. Set $n:= \dim X$ and $k:=
|S|$. For each $p\in S$ let $2p$ be the closed
subscheme of $X$ with $(\Ii _{p,X})^2$ as its ideal sheaf. Set $Z:= \cup _{p\in S} 2p$. We have $h^0(X,L)-h^0(X,\Ii _Z\otimes
L)=(n+1)k$ if and only if for each choice of a degree $2$ zero-dimensional scheme $Z_p\subset 2p$
we have $h^0(X,\Ii _{\cup _{p\in S} Z_p}\otimes L) = h^0(X,L) -2k$.
\end{lemma}

\begin{proof}
The ``only if '' part is trivial, because
if
$W\subseteq A\subset X$ with
$A$ a zero-dimensional scheme, then
$h^0(\Ii _A\otimes L)
\ge h^0(\Ii _W\otimes L)+\deg (W) -\deg (A)$.

Now assume $h^0(X,L)-h^0(X,\Ii _Z\otimes
L)<(n+1)k$. We use $|L|$ to see $X$ as a subscheme of $\PP^r$, $r = h^0(X,L) -1$. For each $p\in S$ call $\mu _p$ the maximal
ideal of the local ring $\Oo _{X,p}$. Since $X$ is smooth and of dimension $n$ at $p$, we have $\dim \mu_p/\mu_p^2 =n$, $\dim \Oo
_{X,p}/{\mu_p}^2 = n+1$ and the vector space $\mu_p/\mu_p^2$ is generated by the images of a system of  $n$ regular parameters of
the local ring
$\Oo _{X,p}$. By the definition of $2p$
the sheaf $(\Ii _{p,X})^2$ is the ideal sheaf of $2p$ as a closed subscheme of $X$. Thus the linear span $\langle 2p\rangle$ in
$\PP^r$ of $2p$ is the Zariski tangent space of $X$ at $p$ and $\dim \langle 2p\rangle =n$. Thus the assumption `` 
$h^0(X,L)-h^0(X,\Ii _Z\otimes L)<(n+1)k$ '' is equivalent to assume that the linear spaces $\langle 2p\rangle$, $p\in S$, are
not $k$ distinct and linearly independent linear subspaces of $\PP^r$. Fix $p\in S$ and take a system of homogeneous
coordinates $x_0,\dots ,x_r$ of $\PP^r$, i.e. a basis of $H^0(X,L)$, such that $p = (1:0:\dots :0)$ and $T_pX$ has $x_i=0$ for
all $i>n$ as its equations. Note that
$x_1/x_0,\dots ,x_n/x_0$ induce a regular system of parameters, $y_1,\dots ,y_n$, of the regular local ring $\Oo _{X,p}$. Fix any line
$D\subset T_pX$ containing $p$. $D$ is induced by the equations $x_i=0$ for all $i>n$ of $T_pX$ and by $n-1$ linearly
independent equations in the variables $x_1,\dots ,x_n$, say $\sum _{i=1}^{n} a_{ij}x_i =0$, $1\le j\le n-1$. Let  $J\subset \mu _p$ be the ideal
generated by the union of all $\sum _{i=1}^{n} a_{ij}y_i \in \mu _p$ and
$\mu_p^2$. $J$ gives a degree $2$ scheme $Z_D =\mathrm{Spec}(\Oo _{X,p}/J)\subset \mathrm{Spec}(\Oo _{X,p}/\mu _p^2) = 2p$, which is the only degree $2$ connected subscheme of $\PP^r$ containing
$p$ and contained in $D$ (and thus spanning $D$). Moreover $Z_D$ is the scheme-theoretic intersection of $2p$ and $D$. It is
also the degree
$2$ connected subscheme of
$D$ with
$p$ as its reduction.

We order the points $p_1,\dots ,p_k$ of $S$ and set $S_i:= \{p_1,\dots ,p_i\}$.
Set $S_0:= \emptyset$. Let $h$ be the last non-negative integer $<k$ such that the linear space $L_h$ spanned by $T_{p_1}X, \dots ,T_{p_h}X$ has dimension $h(n+1)-1$. By the definition of $h$ we have
$T_{p_{h+1}}X\cap L_h \ne \emptyset$. Fix $o\in T_{p_{h+1}}X\cap L_h $.  There is $E\subseteq \{1,\dots ,h\}$
and $o_i\in T_{p_i}X$ such that
$o$ is in the linear span of the set $\{o_i\}_{i\in E}$. 
 For each $p\in S$ the
$k$-dimensional linear subspace $T_pX$ is the union of all lines of $\PP^r$ passing through $p$ and contained in $T_pX$. We
saw the each such line is spanned by a degree $2$ subscheme of $2p$. Thus for each $a\in T_pX$ there is a degree $2$
connected zero-dimensional scheme $W\subset 2p$ such that $a$ is contained in the line $\langle W\rangle$ spanned by $W$; $W$
is unique if and only if $a\ne p$. Thus there are degree $2$ connected zero-dimensional schemes $Z_i\subset 2p_i$, $i \in E\cup \{h+1\}$, such that  $o$ is contained in the line $\langle Z_{h+1}\rangle$ and $o_i$, $i\in E$, is
contained in the line $\langle Z_i\rangle$. If $E\subsetneq \{1,\dots ,h\}$ for all $i\in \{1,\dots ,h\}\setminus E$ fix any
degree $2$ connected scheme $Z_i \subset 2p_i$. Thus $o$ is in the linear span of
the union of the $h$ lines $\cup _{i=1}^{h} \langle Z_i\rangle$. Thus $\cup _{i=1}^{h+1} \langle Z_i\rangle$ spans a linear
space of dimension $<(h+1)n-1$. If $h\le k-2$ call $Z_i$, $i=h+2,\dots ,k$, an arbitrary connected degree $2$ zero-dimensional
scheme
$Z_i\subset 2p_i$.  By construction the lines $\langle Z_i\rangle$, $1\le i\le k$, are either not distinct 
or not linearly independent. 
Thus $h^0(X,L) -h^0(X,\Ii _{\cup Z_i}\otimes L) > 2k$.
\end{proof}

Note that the scheme $Z$ in the statement of Lemma \ref{a1} have degree $k(n+1)$, while the scheme $\cup _{p\in S} Z_p$ has
degree $2k$.

\begin{remark}\label{a2}
Take $(X,L)$ as in Lemma \ref{a1}. An obvious consequence of Lemma \ref{a1} is that $h^0(L)\ge k(n+1)$. In the set-up of
Theorem \ref{i1} (hence with $n=2$) we want to have $h^0(L)>3k$ (not just $h^0(L)\ge 3k$, which would follows from the
$k$-spannedness of $L$ and Lemma \ref{a1}). If $L$ is $k$-spanned and $h^0(L) =3k$, then $\tilde{V}_\delta (L)=\emptyset$ (see the proof of
Theorem \ref{i1}). This is also explicitly stated in the set-up of K3-surfaces in \cite[Remark at page 173]{k1}.
\end{remark}

\begin{proof}[Proof of Theorem \ref{i1}:]
Fix any finite set $S\subset X$ such that $|S| =k$. Set $V(S,L):= \{D\in \tilde{V}_k(L)\mid S=\mathrm{Sing}(D)\}$. Set $Z:=
\cup _{p\in S} 2p$. By Lemma
\ref{a1} we have
$h^0(\Ii _Z\otimes L) =h^0(L) -3k$. Note that $|\Ii
_Z\otimes L|$ parametrizes all curves $D\in |L|$ which are singular at each point of $S$. Thus $V(S,L)=\emptyset$ if
$h^0(L)=3k$, while $V(S,L)$ is a Zariski open subset of a projective space of dimension $h^0(L)-3k-1$ if $h^0(L)>3k$.
Since this is true for all subsets of $X$ with cardinality $k$, we get $\tilde{V}_k(L)=\emptyset$ if $h^0(L) =3k$.
Now assume $h^0(L) >3k$. Since the set of all subsets of $X$ with cardinality $k$ forms an irreducible variety of dimension
$2k$, we would get that
$\tilde{V}_k(L)$ is irreducible of dimension
$h^0(L)-k-1$ if $V(S,L)\ne \emptyset$ for a general $S$. Note that if $V(S,L) =\emptyset$ for a general $S$ with cardinality $k$, then the same holds
for all $S$ with cardinality $k$.
\end{proof}

\begin{remark}
There is a list of pairs $(X,L)$ with $L$ a $k$-very ample line bundle with $L^2\le 4k+4$ (\cite[Table 2]{d} and \cite[Remark
1.4 and note 1]{kn}).
\end{remark}

\providecommand{\bysame}{\leavevmode\hbox to3em{\hrulefill}\thinspace}

\end{document}